\documentclass[a4paper,12pt]{article}
\usepackage[utf8]{inputenc}

\usepackage{amsmath}
\usepackage{amssymb}
\usepackage{amsthm}
\usepackage{mathtools}

\usepackage{cite}

\usepackage{booktabs}

\usepackage{esint}

\usepackage{mathrsfs}
\usepackage{bbm}

\usepackage{accents}

\usepackage{xcolor}
\usepackage{graphicx}        
\usepackage{hyperref}

\title{Well-posedness and time stepping adaptivity for a class of collocation discretisations of time-fractional subdiffusion equations\footnote{This paper has been accepted for publication in Fract. Cal. Appl. Anal. (2026)}}
\author{Sebastian Franz\footnote{
          Institute of Scientific Computing, Technische Universit\"at Dresden, Germany.
          \mbox{e-mail}: sebastian.franz@tu-dresden.de},\and
        Natalia Kopteva\footnote{
          Department of Mathematics and Statistics, University of Limerick, Ireland.
          \mbox{e-mail}: natalia.kopteva@ul.ie}
       }
\date{\today}

\newcommand{\pt}{\partial}
\newcommand{\R}{\mathbb{R}}

\newcounter{EXAMPLEcounter}[section]
\renewcommand{\theEXAMPLEcounter}{\arabic{EXAMPLEcounter}}
\newcommand{\example}{\refstepcounter{EXAMPLEcounter}%
  \textbf{Test example~\theEXAMPLEcounter:}}
\newcommand{\ds}{\mathrm{d}s}
\newcommand{\dsigma}{\mathrm{d}\sigma}
\newcommand{\pmtrx}[1]{\ensuremath{\begin{pmatrix}#1 \end{pmatrix}}}
\newcommand{\W}{\mathcal{W}}
\newcommand{\PS}{\mathbb{P}}
\newcommand{\lrarrow}{\quad\Leftrightarrow\quad}
\newcommand{\ord}[1]{\ensuremath{\mathcal{O}\left(#1\right)}}
\newcommand{\LL}{{\mathcal L}}
\newcommand{\EE}{{\mathcal E}}
\newcommand{\RR}{{\mathcal R}}
\newcommand{\RS } {R_\tau}

\theoremstyle{plain}
\newtheorem{theorem}{Theorem}[section]
\newtheorem{lemma}[theorem]{Lemma}

\newtheorem{corollary}[theorem]{Corollary}
\newtheorem{remark}[theorem]{Remark}

\newcommand{\new}[1]{#1}

\begin{document}
  \maketitle

    \begin{abstract}
      Time-fractional parabolic equations with a Caputo time derivative of order $\alpha\in(0,1)$ are
      discretised in time using collocation methods, which assume that the Caputo derivative of the
      computed solution is piecewise-polynomial.
      For such discretisations of any order $m\ge 0$, with any choice of collocation points, we give
      sufficient conditions for existence and uniqueness of collocation solutions.
      Furthermore, we investigate the applicability and performance of such schemes in the context of the a-posteriori
      error estimation and adaptive time stepping algorithms.
    \end{abstract}



\section{Introduction}\label{sec:intro}

  This paper is devoted to a certain class of collocation discretisations in time for
  time-fractional parabolic equations, of order $\alpha\in(0,1)$, of the form
  \new{
  \begin{subequations}
    \label{eq:problem}
    \begin{align}
      \pt_t^{\alpha}u+\LL u&=f(x,t)&&\mbox{for}\;\;(x,t)\in\Omega\times(0,T],\\
      u(x,t)&=0                    &&\mbox{for}\;\;(x,t)\in\partial\Omega\times(0,T],
    \end{align}
  \end{subequations}
  }%
  also known as time-fractional subdiffusion equations.
  This equation is posed in a bounded Lipschitz domain  $\Omega\subset\R^d$ (where $d\in\{1,2,3\}$),
  subject to an initial condition $u(\cdot,0)=u_0$ in $\Omega$, and the homogeneous boundary condition $u=0$ on $\pt\Omega$ for $t>0$.
  The spatial operator $\LL$  is a linear second-order elliptic operator defined by
  \begin{equation} \label{LL_def}
    \LL u := \sum_{k=1}^d \Bigl\{-\pt_{x_k}\!(a_k(x)\,\pt_{x_k}\!u) + b_k(x)\, \pt_{x_k}\!u \Bigr\}+c(x)\, u,
  \end{equation}
  with sufficiently smooth coefficients $\{a_k\}$, $\{b_k\}$ and $c$ in $C(\bar\Omega)$, for which we assume that $a_k>0$ in $\bar\Omega$,
  and also either $c\ge 0$ or $c-\frac12\sum_{k=1}^d\pt_{x_k}\!b_k\ge 0$.
  The notation $\partial_t^\alpha$ is used for the Caputo fractional derivative in time \cite{Diet10} defined,
  for $t>0$, by
  \[
   \pt_t^{\alpha} u := J_t^{1-\alpha}(\pt_t u),\qquad
    J_t^{\beta} v(\cdot,t) := \frac{1}{\Gamma(\beta)}\int_0^t (t-s)^{\beta-1}v(\cdot,s)\,\ds,
  \]
  where $\Gamma(\cdot)$ is the Gamma function, and $\pt_t$ denotes the classical first-order partial derivative in~$t$.

  The idea  of any collocation scheme is to choose a finite-dimensional space of possible computed solutions
  and a set of collocation points in time, and to impose that the computed solution satisfy \eqref{eq:problem} (or its equivalent version) at the collocation points.
  For example, in \cite{FrK23,FrK25,FrK24b} we considered a family of continuous collocations methods of order $m\ge1$ such that
  \eqref{eq:problem} was satisfied by the continuous piecewise-polynomial computed solution at $m$ distinct collocation points on each time interval.

  In this paper, we shall consider a very different class of collocation methods, which assume that
  the Caputo derivative $\pt_t^\alpha u_\tau$ of the computed solution $u_\tau$ (rather than the computed solution itself)
  is a polynomial of degree $m\ge 0$ in time on each time interval.
  For $m=0$ (and, with the notation below, $\theta_0=1$), such methods were also considered in \cite{LiSal23,LiLiu19}.
  To be more precise, given a fixed $m\geq 0$, for an arbitrary temporal mesh  $0=t_0<t_1<\dots<t_M=T$ with mesh sizes $\tau_k=t_k-t_{k-1}$, define the collocation
  points by $t_k^\ell:=t_{k-1}+\tau_k\,\theta_\ell$ where $0\leq \theta_0<\theta_1<\dots<\theta_{m}\leq 1$.
  Next, define a space $\W_\tau$ of piecewise-polynomial functions
  and the related piecewise-polynomial interpolation operator $\Pi$ to $\W_\tau$ by
  \[
    \W_\tau:=\left\{v\colon v|_{[t_{i-1},t_i]}\in\PS_m(t_{i-1},t_i) \right\},
    \quad
    \bigl(w-\Pi w\bigr)(t_k^\ell)=0\;\;\forall\, \{t_k^\ell\} .
  \]
  Now we request that our computed solution $u_\tau$ satisfy the following semi-dis\-cre\-ti\-sa\-tion  of \eqref{eq:problem} in time:
  \begin{equation}\label{eq:coll_def}
    \pt_t^{\alpha}u_\tau+ \Pi (\LL u_\tau- f)=0\qquad \mbox{in~}\Omega\times(0,T],
  \end{equation}
  subject to the same initial and boundary conditions as imposed on \eqref{eq:problem}.
  Importantly, \eqref{eq:coll_def} implies that $\pt_t^{\alpha}u_\tau$, as a function of $t$, is in $\W_\tau$, while $u_\tau$ is not a polynomial even on $(0,t_1)$.
  If $\theta_0=0$ and $\theta_1=1$, the piecewise-polynomial $\pt_t^{\alpha}u_\tau$ becomes continuous in time, so we refer to the resulting method
  as  a \textit{continuous collocation} method; otherwise, $\pt_t^{\alpha}u_\tau$ is discontinuous at $t=t_k$, $k\ge1$,
  and the resulting method will be referred to as a \textit{discontinuous collocation} method
  (while $u_\tau$ is always continuous in time).

  Note that both the original equation \eqref{eq:problem} and
  the collocation scheme \eqref{eq:coll_def} allow an alternative equivalent representation, obtained by an application of $J_t^\alpha$
  to \eqref{eq:problem}  and \eqref{eq:coll_def}, respectively, which
  (combined with the well-known identity $J_t^\alpha\pt_t^{\alpha}v(t)=v(t)-v(0)$) yields
  $u+ J_t^\alpha(\LL u -f)=u_0$ and its discretisation in time
  \begin{equation}\label{eq:coll_def1}
    u_\tau+J_t^\alpha[\Pi (\LL u_\tau- f)]=u_0\qquad \mbox{in~}\Omega\times(0,T].
  \end{equation}
  Interestingly, in the simpler case of $\LL= c$ (i.e. without spatial derivatives in \eqref{eq:problem}),
  where $c$ is a constant, the original problem becomes
  $u+ c\, J_t^\alpha u =u_0+J_t^\alpha f$, i.e. a Volterra weakly singular equation of the second kind, while
  \eqref{eq:coll_def1} can be rewritten as
  $\Pi u_\tau+ c \,J_t^\alpha[\Pi  u_\tau]= u_0 +J_t^\alpha[\Pi f]$ $\forall\, t\in\{t_k^\ell\}$,
  which shows that $\Pi u_\tau\in \W_\tau$ is
  a standard collocation solution, considered, e.g., in \cite[Section 6.2]{Brunner04}
  (with a minor change in that $J_t^\alpha f$ is approximated by 
  $J_t^\alpha[\Pi f]$  in \eqref{eq:coll_def1}).

  Another equivalent representation of \eqref{eq:coll_def}, which may be particularly convenient for implementation,
  is obtained by introducing an auxiliary function $w_\tau\in\W_\tau$ as follows:
  \begin{align}\label{eq:coll_def_w}
    w_\tau+\Pi[ \LL J_t^\alpha w_\tau-f]+\LL u_0 =0,\qquad \pt_t^{\alpha}u_\tau= w_\tau\qquad \mbox{in~}\Omega\times(0,T],
  \end{align}
  where we also used $ u_\tau-u_0=J_t^\alpha\pt_t^{\alpha}u_\tau=J_t^\alpha w_\tau$.
  Note one apparent difference between \eqref{eq:coll_def1} and \eqref{eq:coll_def_w}
  in that the latter involves $\LL u_0$ (see Remark~\ref{rem_u0_nonsmth} for a further discussion).

  We also note the similarity of the discrete formulations \eqref{eq:coll_def1} and \eqref{eq:coll_def_w}, 
  as  they can respectively be rewritten at the collocation points as
  \begin{align}\label{eq:coll_def1_col}
    \Pi u_\tau+\LL J_t^\alpha (\Pi  u_\tau)&=u_0+J_t^\alpha(\Pi f)\qquad \mbox{in~}\Omega\times\{t_k^\ell\}, \\
    \label{eq:coll_def_w_col}
       w_\tau+ \LL J_t^\alpha w_\tau&=f-\LL u_0\qquad\qquad \mbox{in~}\Omega\times\{t_k^\ell\},
  \end{align}
  where we used $u_\tau=\Pi u_\tau$ $\forall\,\{t_k^\ell\}$ when rewriting \eqref{eq:coll_def1}, and
  $\Pi[ \LL J_t^\alpha w_\tau-f]=\LL J_t^\alpha w_\tau-f$ $\forall\,\{t_k^\ell\}$ when rewriting \eqref{eq:coll_def_w}.
  Thus, 
  both $\Pi  u_\tau$ and $w_\tau$ are solutions of the discrete
  \new{integral}
  equation with the operator
  $I+\LL J_t^\alpha$ and an appropriate right-hand side, subject to the homogeneous Dirichlet conditions on $\pt \Omega$.

  While there is substantial literature on collocation schemes of type \eqref{eq:coll_def} in the context of Volterra
  weakly singular equations of the second kind; see e.g. \cite[Section~6.2]{Brunner04} or \cite{LiangStynes24}%
  ---see also \cite{LPSV23}, where, essentially, systems of Volterra weakly singular equations of the second kind are
  considered%
  ---there are very few papers that address subdiffusion equation of type \eqref{eq:problem}, or its equivalent version
  $u+ J_t^\alpha(\LL u -f)=u_0$,  in this context.%
  \footnote{While \cite{LPSV23} is devoted to a 1d subdiffusion equation of type \eqref{eq:problem}, which is discretised
  in space using finite differences, and in time using  collocation schemes of type~\eqref{eq:coll_def_w}, the error constants
  in the a-priori error bounds depend on the matrix associated with the spatial discretisation
  (or, equivalently, on the spatial mesh size). Hence, essentially, the error analysis in \cite{LPSV23} addresses
  \eqref{eq:coll_def_w} in the context of a system of Volterra weakly singular equations of the second kind.}
  The case $m=0$ and $\theta_0=1$ for equations of type \eqref{eq:problem} was addressed in \cite{LiSal23,LiLiu19},
  the notable feature of this case being that the discretisation of $\pt_t^\alpha$  is associated with an M-matrix.

  In this paper, we address both discontinuous and continuous collocation schemes of type \eqref{eq:coll_def} for any $m\ge 0$ and any sets of collocation points.
  We establish the unique solvability for such schemes (along the lines of our earlier work \cite{FrK24b}) and, furthermore, investigate
  the applicability and performance of such schemes in the context of the a-posteriori error estimation and
  adaptive time stepping algorithms developed in  \cite{Kopteva22,Kopt_Stynes_apost22,FrK23}.

  The paper is organised as follows.
  In Section~\ref{sec:wellposed} we establish the unique solvability for discontinuous collocation schemes \eqref{eq:coll_def} with $\theta_0>0$.
  both continuous collocation schemes ($\theta_0=0$ and $\theta_1=1$) and discontinuous collocation schemes ($\theta_0=0$ and $\theta_1<1$).
  In Section~\ref{sec:res} we give residual-type a-posteriori estimates for the errors induced by discretisation in time,
  which naturally lead to adaptive time stepping algorithms with local time step criteria (while the original subdiffusion equation is non-local).
  We also describe computationally stable implementations of the considered methods, as well as a stable computation of the residuals.
  Finally, in Section~\ref{sec:numerics} we report numerical experiments, which support the applicability and reliability of the considered time stepping
  algorithm for subdiffusion equations; we also compare the resulting errors to those of another class of collocation methods considered in \cite{FrK23}.

  \textit{Notation.}\,
  We use  the standard inner product $\langle\cdot,\cdot\rangle$
  in the space $L_2(\Omega)$, as well as the standard spaces 
  $L_\infty(\Omega)$ and $H^1_0(\Omega)$, the latter denoting the space of functions in the
  Sobolev space $W_2^1(\Omega)$ vanishing on $\pt\Omega$, as well as its dual space $H^{-1}(\Omega)$.
  The notation of type $(L_2(\Omega))^n$ is used for the space of vector-valued functions
  $\Omega\rightarrow \R^n$ with all vector components in $L_2(\Omega)$, and the norm induced
  by $\|\cdot\|_{L_2(\Omega)}$ combined with any fixed norm in $\R^n$.
  We also use standard spaces
  $C([0,T];\,L_p(\Omega))$, and $W^1_\infty(t',t'';\,L_p(\Omega))$ for $p\in\{2,\,\infty\}$ for functions of $x$ and~$t$;
  see \cite[Section 5.9.2]{evans}.

\section{Well-posedness for discontinuous collocation schemes with $\theta_0>0$}\label{sec:wellposed}
  In this section, we shall establish the unique solvability for collocation schemes \eqref{eq:coll_def},
  which allow the equivalent representations \eqref{eq:coll_def1_col} and \eqref{eq:coll_def_w_col}
  at the collocation points.
  Notably, both these representations are discrete
  integro-differential equations with the same operator $I+\LL J_t^\alpha$, applied to a piecewise-polynomial-in-time solution (which, as a function of $t$, is in $\W_\tau$), and an appropriate right-hand side.

  Following \cite[Section 2]{FrK24b}, it suffices to consider only the first time interval $(0,t_1)$ (as the other intervals can be treated similarly), and furthermore, rescale the collocation scheme on this interval to the time interval $(0,1)$.
  In other words, it suffices to investigate the well-posedness of the problem
  \begin{equation}\label{eq:problemRescaled}
    U(x, \theta_\ell)+  \tau^\alpha \LL J_t^\alpha\,U(x, \theta_\ell)=F(x, \theta_\ell)\qquad \mbox{for}\;\; x\in\Omega,\;\;  \ell\in\{0,\ldots, m\},
  \end{equation}
  where
  $\tau:=\tau_1$ (or $\tau:=\tau_k$ when considering the interval $(t_{k-1},t_k)$), while
  $U(\cdot, t)$ is a polynomial of degree $m$ in time for $t\in(0,1]$ subject to
  $U=0$ on $\pt\Omega$.
  To be more precise,
  $U(\cdot, t)=\Pi u_\tau(\cdot ,t_{k-1}+ \tau_k t)$
  if \eqref{eq:problemRescaled} is associated with \eqref{eq:coll_def1_col} on $(t_{k-1},t_k)$,
  or
  $U(\cdot, t)=w_\tau(\cdot ,t_{k-1}+\tau_k t)$ if \eqref{eq:problemRescaled} is associated with \eqref{eq:coll_def_w_col}.
  The appropriate right-hand side $F(\cdot, t)$ in \eqref{eq:problemRescaled}  involves the collocation solution for $t\le t_{k-1}$ and
  depends on the considered formulation
  \eqref{eq:coll_def1_col} or \eqref{eq:coll_def_w_col}.

\subsection{Matrix representation of the collocation scheme \eqref{eq:problemRescaled}}\label{ssec_matrix}
  A solution $U$ of \eqref{eq:problemRescaled}, being a polynomial of degree $m$ in time, can be represented
  using the monomial basis $\{t^j\}_{j=0}^m$ as
  \begin{equation}\label{matr1}
    U(x,t)=\sum_{j=0}^m v_j(x)\,t^j = \Bigl(1,\,t,\,\cdots,\,t^m\Bigr)\pmtrx{v_0(x)\\v_1(x)\\\vdots\;\;\\v_m(x)}=:T(t)\, \vec{V}(x).
  \end{equation}
  Here we use the standard linear algebra multiplication, and, for convenience, we  highlight column
  vectors in such evaluations with an arrow (as in $\vec{V}$).
  Other column vectors of interest are
  $\vec{\theta}:=\big(\theta_0,\theta_1,\cdots,\theta_m\big)^{\!\!\top}$,
  and also $\vec{U}(x):=U(x,\vec{\theta})$ and $\vec{F}(x):=F(x,\vec{\theta})$
  (where a function is understood to be applied to a vector argument elementwise).

  Next, we represent $\vec{U}=U(\cdot,\vec{\theta})$ using
  the Vandermonde-type matrix $W$ as follows:
  \begin{equation}\label{matr2}
    W:=T(\vec{\theta})=\pmtrx{T(\theta_0)\\T(\theta_1)\\\vdots\;\;\;\\T(\theta_m)}=
       \pmtrx{1 & \theta_0 & \cdots &\theta_0^m\\
              1 & \theta_1 & \cdots &\theta_1^m\\
              \vdots\ &\vdots\ && \vdots\ \\
              1 & \theta_m &\cdots & \theta_m^m}
    \quad \Rightarrow\,
    \vec{U}(x)=U(x,\vec{\theta})
    =W \vec{V}(x).
  \end{equation}

  Before rewriting \eqref{eq:problemRescaled} for $\vec{U}(x)=U(x,\vec{\theta})$,
  we need a matrix representation of $J_t^\alpha U(x,\vec{\theta})$, for which
  the following two diagonal matrices will be useful:
  \begin{equation}\label{matr3}
    D_1:={\rm diag}(\theta_0^{\alpha},\cdots,\theta_m^{\alpha}),\;\;
    D_2:={\rm diag}(c_0,\cdots,c_m),\;\,
    c_j:=\frac{\Gamma(j+1)}{\Gamma(j+1+\alpha)}.
  \end{equation}
  Then we have
  \begin{align*}
    J_t^\alpha U(x,t)
      &= \sum_{j=0}^m v_j(x)\, J_t^\alpha t^j
       = \sum_{j=0}^m v_j(x) c_j t^{j+\alpha}
       = t^\alpha\, T(t) D_2 \vec V(x).
  \end{align*}
  Finally, $J_t^\alpha U(x,\vec{\theta})= D_1T(\vec{\theta})  D_2\,\vec{V}(x)=D_1W D_2\,\vec{V}(x)$.
  Hence, \eqref{eq:problemRescaled} is equivalent to
  \begin{equation}\label{VV_version}
    \Bigl(W +\tau^\alpha \LL \,D_1W  D_2\Bigr)\,\vec{V}(x)=\vec{F}(x).
  \end{equation}
  Assuming $\theta_0>0$ (see Section~\ref{sec:continuous} for modifications in the case of $\theta_0=0$),
  we set $\vec{Y}(x):=D_2\vec{V}(x)$, which yields the following version of our collocation scheme \eqref{eq:problemRescaled}:
  \begin{gather}\label{Y_version}
    \Bigl(D_1^*W D_2^*+\tau^\alpha \LL \,W  \Bigr)\,\vec{Y}(x)=\vec{F}^*(x), \qquad\mbox{where~}\theta_0>0,
  \end{gather}
  where we use the diagonal matrices $D_1^*:=D_1^{-1}$ and $D_2^*:=D_2^{-1}$, and the right-hand side vector $\vec{F}^*(x):=D_1^{-1}\vec{F}(x)$.
  The above representation \eqref{Y_version} is of type \cite[(10)]{FrK24b}, albeit with somewhat different matrices $D_1^*$, $D_2^*$, and $W$;
  hence, both approaches used in \cite{FrK24b}, the Lax-Milgram Theorem and eigenvalue tests, will be employed below to show the well-posedness of
  \eqref{eq:problemRescaled}.

  \begin{remark}\label{rem_inapplble}[ODE system vs. subdiffusion equation]
    As discussed in \cite[Section 2.1]{FrK24b}, if  \eqref{eq:problem} is a system of $n$ fractional-order ordinary differential equations, with
    a linear operator $\LL: \R^n\rightarrow \R^n$, the norm of which is $\|\LL\|_\star$ (induced by any vector norm in $\R^{n}$), then
    to ensure that \eqref{Y_version} has a unique solution, it suffices to assume that
    \[
      \tau^\alpha\|\LL\|_\star<\|W^{-1} (D_1^*W D_2^*)\|,
    \]
    where the matrix norm $\|\cdot\|$ is induced by any vector norm in $\R^{m+1}$.
    However, the above sufficient condition becomes very restrictive if $\LL=\LL_h$ is a finite-dimensional discrete operator approximating an elliptic operator.
    For example, for the standard finite difference discretisation $\LL_h$ of the differential operator $-\partial_x^2$ in the domain $\Omega=(0,1)$,
    it is well-known that $\|\LL\|_\star=O( DOF^{2})$, where $DOF$ is the number of degrees of freedom in space, so one would need to impose
    a very restrictive condition
    $\tau^\alpha DOF^{2}\le C_m$, with $C_m$ depending on $m$ and $\{\theta_{\ell}\}$.
    Hence, in this paper we explore alternative approaches to establishing the well-posedness of \eqref{Y_version}.
  \end{remark}

\subsection{Well-posedness by means of the Lax-Milgram Theorem}\label{ssec_Lax}
  As representation \eqref{Y_version} is of type \cite[(10)]{FrK24b},
  and additionally the condition $c-\frac12\sum_{k=1}^d\pt_{x_k}\!b_k\ge 0$ implies that the operator $\LL$ in \eqref{eq:problem} is associated with a coercive bilinear form,
  we can immediately apply the following result.

  \begin{theorem}[\!\!{\cite[Theorem~5]{FrK24b}}]\label{the_LaxM}
    Suppose that the bilinear form $a(v,w):=\langle \LL v, w\rangle$ on the space $H^1_0(\Omega)$ is bounded and coercive.
    Additionally, suppose that there exists a diagonal matrix $D\in \R^{m+1,m+1}$ with strictly positive diagonal elements such that
    the symmetric matrix $W^{\top}D WD^*_2+(W^{\top}D WD^*_2)^{\top}$ is positive-semidefinite.
    Then \eqref{Y_version} is associated with a bounded and coercive
    bilinear form,
    \new{
    \begin{equation*}
      A (\vec V, \vec V^*)=
      \langle W^{\top} D WD_2^*\,\vec{V}, \vec V^* \rangle+\tau^\alpha\langle W^{\top} D \LL W\vec V,\vec V^* \rangle
    \end{equation*}
    }%
    on the space $(H^1_0(\Omega))^{m+1}$. If, additionally, $\vec F^*\in (H^{-1}(\Omega))^{m+1}$, then
    there exists a unique solution to problem \eqref{Y_version}. Equivalently, if
    $\{F(\cdot, \theta_\ell)\}_{\ell=0}^m\in (H^{-1}(\Omega))^{m+1}$, then
    \eqref{eq:problemRescaled} has a unique solution $\{U(\cdot, \theta_\ell)\}_{\ell=0}^m$
    in $(H^1_0(\Omega))^{m+1}$.
  \end{theorem}
  
  For the lowest order cases \cite[Corollary 8]{FrK24b} provides a blueprint on finding such a matrix $D$.
  A calculation shows that for the trivial case  $m=0$ and $\theta_0>0$, one uses the matrix $D=1\in\R^{1,1}$, while for the case
  \begin{equation}\label{m1_lax}
    m=1,\qquad 0<\frac{\theta_0}{\theta_1}\leq \frac{1}{1+\alpha},
  \end{equation}
  Theorem~\ref{the_LaxM} applies with the matrix $D=\mathrm{diag}\{\theta_1,\theta_0\}$.
  For higher-order collocation schemes the construction of a suitable $D$ becomes more intricate,
  so semi-computational techniques may be useful for particular sets of collocation points $\{\theta_{\ell}\}$. Alternatively,
  one can employ an eigenvalue test considered in the next section.

  \begin{corollary}[Galerkin finite element discretisation]
    The existence and uniqueness results of Theorem~\ref{the_LaxM}
    remain valid if $(H^1_0(\Omega))^{m+1}$ is replaced by $(S_h)^{m+1}$, where $S_h$ is a finite-dimensional subspace of
    $H^1_0(\Omega)$.
    Equivalently, they apply to the corresponding spatial discretisation of \eqref{Y_version}.
  \end{corollary}

\subsection{Well-posedness by means of an eigenvalue test}\label{ssec_eigen}

  Throughout this section, we assume that $\LL$ has a set of real positive eigenvalues
  $0<\lambda_1\le\lambda_2\le \lambda_3\le\ldots$ and a corresponding basis of orthonormal eigenfunctions
  $\{\psi_n(x)\}_{n=1}^\infty$.
  This assumption is immediately satisfied if $\LL$ in \eqref{LL_def} is
  self-adjoint (i.e. $b_k=0$ for $k=1,\ldots,d$) and $c\ge 0$; see, e.g., \cite[Section~6.5.1]{evans}  for further details.
  (Otherwise, even if $\LL$ is non-self-adjoint, the analysis can sometimes,  as, e.g., for
  $\LL = -\sum_{k=1}^d \bigl\{\pt^2_{x_k} + (\pt_{x_k}\! B(x))\, \pt_{x_k} \bigr\}+c(x)$
  \cite[problem~8.6.2]{evans}, be reduced to the self-adjoint case.)

  With the above assumption, the well-posedness of problem \eqref{Y_version} can be investigated using
  an eigenfunction expansion of its solution $\vec{Y}(x)$.
  Indeed, using the eigenfunction expansions $\vec{Y}(x)=\sum_{n=1}^\infty \vec{Y}_n\, \psi_n(x)$
  and $\vec{F}^*(x)=\sum_{n=1}^\infty \vec{F}^*_n\, \psi_n(x)$, one reduces \eqref{Y_version} to
  the following matrix equations for vectors $\vec{Y}_n$ $\forall\,n\ge1$:
  \begin{equation}\label{lambda_n_prob}
    \Bigl(D_1^*W D_2^*+\tau^\alpha \lambda_n \,W  \Bigr)\,\vec{Y}_n=\vec{F}_n^*\,.
  \end{equation}

  The solvability of \eqref{lambda_n_prob}, and thus of \eqref{Y_version}, can be investigated with the help of the following lemma.
  \begin{lemma}[{\!\!\cite[Lemma 11]{FrK24b}}]\label{lem_M}
    For any fixed invertible matrix $M\in\R^{m+1,m+1}$, consider $R(\lambda\,;M):=(M+\lambda I)^{-1}$ for $\lambda\in \R$. Then
    $\lim_{\lambda\rightarrow+\infty}\| R(\lambda\,;M)\|=0$.
    Furthermore, there is a constant $C=C(M)>0$ such that
    \begin{equation}\label{resolvent}
      0<\| R(\lambda\,;M)\|\le C\;\;\;\forall\lambda\ge 0\;\;\Leftrightarrow\;\;
      \mbox{$M$ has no negative real eigenvalues}
    \end{equation}
    (where  $\|\cdot\|$ denotes the  matrix norm induced by 
    any vector norm in $\R^{m+1}$).
  \end{lemma}

  \begin{proof}
    We shall outline the proof for completeness. First, for any fixed $M$, once $\lambda>0$ is sufficiently large, $M+\lambda I$ is invertible,
    while $R(\lambda\,;M)=\lambda^{-1}(\lambda^{-1}M+I)^{-1}\rightarrow \lambda^{-1} I$, so $\| R(\lambda\,;M)\| \rightarrow 0$, as $\lambda\rightarrow+\infty$.
    To prove $\Rightarrow$ in \eqref{resolvent}, note that if $M$ has a real negative eigenvalue $-\lambda^*<0$, then  $\| R(\lambda^*\,;M)\|$ is not defined.
    Finally, to prove  $\Leftarrow$, note that $M+\lambda I$ is invertible for any $\lambda\le 0$,
    so $\| R(\lambda\,;M)\|>0$ $\forall\, \lambda\ge 0$. As $\| R(\lambda\,;M)\|$ is a continuous positive function
    of $\lambda$ on $[0,\infty)$, which decays as $\lambda\rightarrow +\infty$, it has to be bounded by a positive constant $C=C(M)$.
  \end{proof}

  As all matrices $D_1$, $D_2$, and $W$ are invertible for $\theta_0>0$, in view of Lemma~\ref{lem_M},
  for the unique solvability of \eqref{lambda_n_prob}, it suffices to check that
  $M:=W^{-1}D_1^*W D_2^*=W^{-1}D_1^{-1}W D_2^{-1}$ has no negative real eigenvalues, which leads us to the following result.

  \begin{corollary}\label{cor_M_matrix}
    Suppose that the operator $\LL$ has a set of real positive eigenvalues
    and a corresponding basis of orthonormal eigenfunctions,
    and $\vec F(x)$ is sufficiently smooth.
    If, additionally, the matrix $M=W^{-1}D_1^{-1}W D_2^{-1}$ has no negative real eigenvalues, then
    there exists a unique solution to problem \eqref{Y_version}, or, equivalently,
    to problem \eqref{eq:problemRescaled}.
  \end{corollary}

  It remains to check the hypothesis on $M$ made in the above corollary.

  \begin{theorem}\label{thm:existCollDisc}
    For any $m\geq 0$ and any set of collocation points $\{\theta_\ell\}_{\ell=0}^m$
    with $\theta_0> 0$, the matrix $M=W^{-1}D_1^{-1}W D_2^{-1}$ has no negative real eigenvalues, so
    the well-posedness of \eqref{eq:problemRescaled} follows from Corollary~\ref{cor_M_matrix}.
  \end{theorem}

  \begin{proof}
    It holds for any eigenvalue $\lambda$ of $M$ that
    \begin{align*}
      \det(M-\lambda I)=0
       &\lrarrow
      \det(WD_2^{-1}-\lambda D_1W)=0.
    \end{align*}
    The two matrices here are given by
    \begin{align*}
      M_1:=WD_2^{-1} &= \pmtrx{[c]1&\theta_0&\cdots&\theta_0^m\\
                              \vdots&\ddots &&\vdots\\
                              1&\theta_m&\cdots&\theta_m^m}
                        \pmtrx{c^*_0\\&\ddots\\&&c^*_m}
                     = \pmtrx{[c]c^*_0\; & c^*_1\theta_0 & \cdots &c^*_m\theta_0^m\\
                              c^*_0\; & c^*_1\theta_1 &       & \vdots\\
                               \vdots    &                    &\ddots &\vdots\\
                              c^*_0\; &\cdots&\cdots &c^*_m\theta_m^m},
    \intertext{where $c^*_j:=(c_j)^{-1}=\frac{\Gamma(j+1+\alpha)}{\Gamma(j+1)}>0$, and}
      M_2:=D_1W &= \pmtrx{\theta_0^\alpha\\&\ddots\\&&\theta_m^\alpha}
                   \pmtrx{[c]1&\theta_0&\cdots&\theta_0^m\\
                              \vdots&\ddots &&\vdots\\
                              1&\theta_m&\cdots&\theta_m^m}
                            = \pmtrx{[c]\theta_0^\alpha & \theta_0^{\alpha+1} & \cdots & \theta_0^{\alpha+\new{m}}\\
                                     \theta_1^\alpha & \theta_1^{\alpha+1} &       & \vdots\\
                                     \vdots          &                     & \ddots& \vdots\\
                                     \theta_m^\alpha & \cdots               & \cdots & \theta_m^{\alpha+m}}.
    \end{align*}
    Now, to show that the characteristic polynomial
    \[
      \det(M_1-\lambda M_2)=\sum_{j=0}^{m+1} (-\lambda)^j a_j
    \]
    has no negative real root, it suffices to check that $a_j>0$ $\forall, j=0,\ldots, m+1$.
    These coefficients can be represented via determinants of certain $(m+1)\times (m+1)$ matrices, denoted
    by $M_{\mathcal I}$, where $\mathcal I$ is a subset of  $\{1,\ldots,m+1\}$, constructed column by column,
    using the corresponding $k$th column of $M_1$ or $M_2$:
    \[
      M_{\mathcal I}^{k} := \begin{cases}
                    M_1^k & \mbox{if}\;\;k\not\in {\mathcal I}\\
                    M_2^k & \mbox{if}\;\; k\in \mathcal I
                  \end{cases}
                  \qquad\mbox{for any}\;\;
                  {\mathcal I}\subseteq \{1,\ldots, m+1\},\quad k=1,\ldots, m+1.
    \]
    For example, for $a_0$ and $a_{m+1}$ we immediately have:
    \[
      a_{m+1} = \det M_2 = \det M_{\{1,\ldots,m+1\}}
      \quad\text{and}\quad
      a_0 = \det M_1 = \det M_{\emptyset}.
    \]
    For the other coefficients, note that if the $k$th column $A^k$ of some matrix $A$ allows a representation
    $A^k=B^k-\lambda C^k$ for some column vectors $B^k$ and $C^k$, then
    \begin{multline*}
       \det\Bigl(A^1 \cdots A^{k-1}A^kA^{k+1}\cdots A^{m+1}\Bigr)\\
       =
       \det\Bigl(A^1 \cdots A^{k-1}B^kA^{k+1}\cdots A^{m+1}\Bigr)\\
       -\lambda\,\det\Bigl(A^1 \cdots A^{k-1}C^kA^{k+1}\cdots A^{m+1}\Bigr),
    \end{multline*}
    which is easily checked using the $k$th column expansion.
    Using this property for all columns in $M_1-\lambda M_2$, one gets
    \begin{equation}\label{eq:poly2}
      \det(M_1-\lambda M_2)=
      \sum_{j=0}^{m+1} (-\lambda)^{j}\,\Bigl\{\sum_{\#{\mathcal I}=j}\! \det M_{\mathcal I}\Bigr\},
    \end{equation}
    where $\#{\mathcal I}$ denotes the number of elements in any index subset $\mathcal{I}$.

    Finally, we show that these remaining determinants are all non-negative. It holds for any $\mathcal{I}$
    \begin{align*}
      \det M_{\mathcal{I}}
        = \Bigl(\prod_{k\not\in {\mathcal I}} c^*_k \Bigr)\,
         \det\pmtrx{[c]\theta_0^{\beta_0} & \theta_0^{\beta_1} & \cdots &\theta_0^{\beta_{m}}\\
                       \theta_1^{\beta_0} & \theta_1^{\beta_1} & \cdots &\theta_1^{\beta_{m}}\\
                       \vdots             & \vdots             & \ddots & \vdots\\
                       \theta_m^{\beta_0} & \theta_m^{\beta_1} & \cdots &\theta_m^{\beta_m}}
                      ,
            \quad
            \beta_k:= \begin{cases}
                    k & \mbox{if}\;\;k+1\not\in {\mathcal I},\\
                    k+\alpha &\mbox{if}\;\; k+1\in \mathcal I.
                  \end{cases}
    \end{align*}
    We have $c^*_k>0$ for any $k$ and, therefore, the product in the above representation is positive.
    \new{The remaining determinant is the determinant of a so-called generalised Vandermonde matrix \cite{Heineman29},
    the positivity of which is established in \cite{RS00, YWZ01}
    under the condition that the sequences $\beta_j$ and $\theta_j$ are strictly increasing with $\theta_0>0$,
    while we, indeed, have}
    \[
      0\leq\beta_0<\beta_1<\dots<\beta_m
      \quad\text{and}\quad
      0<\theta_0<\theta_1<\dots<\theta_m\le 1.
    \]
    The desired assertion that $\det M_{\mathcal I}>0$ follows, which completes the proof.
  \end{proof}

  \begin{remark}[classical parabolic equations]
    In the limit case $\alpha=1$ all fractional operators become the classical ones. The proof of
    existence of our method still applies. But, now we only have $0\leq\beta_0\leq\beta_1\leq\dots\leq\beta_m$,
    and if there is one pair of equal powers, the corresponding determinant $\det M_{\mathcal I}$ is zero.
    Nevertheless, the sum $\sum_{\#{\mathcal I}=j}\! \det M_{\mathcal I}$ in \eqref{eq:poly2}
    for each $j\in\{1,\ldots, m+1\}$ includes exactly one positive determinant,
    which corresponds to ${\mathcal I}=\{m-j+2, \ldots, m+1\}$ (understood as ${\mathcal I}=\emptyset$ if $j=0$).
    Hence, each such sum is positive, so $a_j>0$ $\forall\, j\in\{0,\ldots,m+1\}$, so
    Theorem~\ref{thm:existCollDisc} remains valid for the classical case $\alpha=1$ with any $m\ge 0$
    and any set of collocation points $0\leq\theta_0 <\theta_1< \dots < \theta_m\leq 1$.
    Likewise, the analysis of the next Section~\ref{sec:continuous} (for the case $\theta_0=0$, which includes continuous collocation schemes)
    applies to the classical case $\alpha=1$.
    \new{Note also that for $\alpha=1$,
    the considered collocation methods \eqref{eq:coll_def} reduce
    to implicit Runge-Kutta schemes
    (in the case of Gauss-Legendre points, they are also equivalent to continuous Galerkin time-stepping schemes with quadrature), for which well-posedness results are available in the literature.
    }
  \end{remark}

\subsection{A semi-computational approach}
  In view of the role of the matrix  $M=W^{-1}D_1^{-1}W D_2^{-1}$ in Corollary~\ref{cor_M_matrix},
  one may also investigate the eigenvalues of $M$ for any given $m$ and $\{\theta_\ell\}$ numerically
  (following the numerical procedure in \cite[Section~4.2]{FrK24b}).
  While in Theorem~\ref{thm:existCollDisc} we have already established that $M$ never has negative real eigenvalues, such numerical investigation
  is still helpful, as it additionally shows how far the spectrum of $M$ is located from the negative real axis.

  To illustrate this approach, in Figures~\ref{fig:exist_eq} and \ref{fig:exist_gaussleg}
  \begin{figure}
    \begin{center}
      \includegraphics{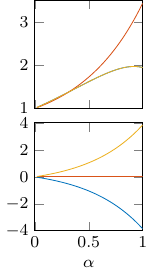}
      \includegraphics{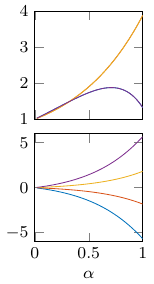}
      \includegraphics{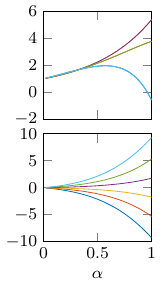}
      \includegraphics{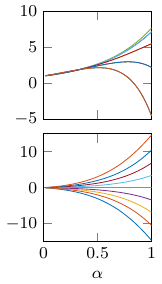}
    \end{center}\vspace{-0.6cm}
    \caption{Eigenvalues for collocation methods using equidistributed points for $m\in\{2,3,5,8\}$ (left to right) and real parts (top), imaginary parts (bottom)\label{fig:exist_eq}}
  \end{figure}
  \begin{figure}
    \begin{center}
      \includegraphics{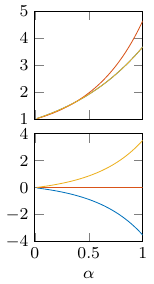}
      \includegraphics{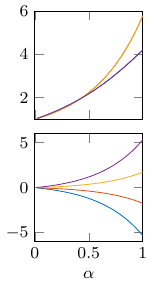}
      \includegraphics{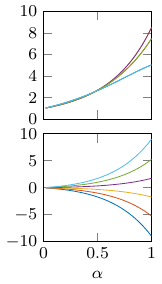}
      \includegraphics{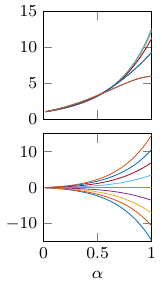}
    \end{center}\vspace{-0.6cm}
    \caption{Eigenvalues for collocation methods using \new{Gauss}-Legendre points for $m\in\{2,3,5,8\}$ (left to right) and real parts (top), imaginary parts (bottom)\label{fig:exist_gaussleg}}
  \end{figure}
  we plot the eigenvalues of $M$ for all $\alpha\in(0,1]$ and $m\in\{2,3,5,8\}$ for two sets of collocation points: $\theta_\ell=\frac{\ell+1}{m+2}$
  and the \new{Gauss}-Legendre points.
  We observe that in all cases there are no negative real eigenvalues.
  Furthermore, for the \new{Gauss}-Legendre family of points,
  both Fig.\,\ref{fig:exist_gaussleg}  and further numerical experiments, performed $\forall\,m\le20$,
  indicate that the real parts of all eigenvalues are positive. Not only these numerical results support Theorem~\ref{thm:existCollDisc},
  but they provide additional information about the spectrum of $M$.

\section{Well-posedness for collocation schemes with $\theta_0=0$}\label{sec:continuous}

  In this section we shall extend the results of the previous section to the case of $\theta_0=0$.
  This includes the case of $\theta_0=0$ and $\theta_m=1$, which---as discussed in \cite[Section 2.2.2]{Brunner04}---yields
  continuous collocation schemes, as well as the case of $\theta=0$ and $\theta_m<1$, which yields
  discontinuous collocation schemes.

  Recalling the definition \eqref{eq:coll_def} of collocation schemes under consideration,
  note that $u_\tau$ is always continuous in time, while $\pt_t^\alpha u_\tau$ and $\Pi u_\tau$ are in $\W_\tau$ as functions of time, i.e. may be discontinuous,
  so for such functions we shall use the notation of type $\Pi u_\tau(\cdot, t_k^-)$ and $\Pi u_\tau(\cdot, t_k^+)$.

  Now, suppose that $\theta_0=0$.
  Then we have $(\Pi v)(t_k^+)=v(t_k)$ for any continuous $v(t)$.
  So \eqref{eq:coll_def} yields
  $\pt_t^{\alpha}u_\tau(\cdot, t_{k}^+) + (\LL u_\tau- f)(\cdot, t_{k})=0$.
  If, additionally, $\theta_m=1$ (continuous collocation),
  then we additionally have $(\Pi v)(t_k^-)=v(t_k)$, and so $\pt_t^{\alpha}u_\tau(\cdot, t_{k}^-) + (\LL u_\tau- f)(\cdot, t_{k})=0$,
  i.e. we effectively use \new{the space} $\W_\tau^{cont}:=\W_\tau\cap C([0,T])$.

  In terms of well-posedness, whether $\theta_m=1$ or $\theta_m<1$, let us investigate \eqref{eq:coll_def}, or its equivalent formulation
  \eqref{eq:coll_def_w_col},  on $(t_{k-1},t_k]$.
  Assuming that $u_\tau(\cdot, t)$ is already computed for $t\le t_{k-1}$,
  note that $w_\tau(\cdot, t_{k-1}^+) = \pt_t^{\alpha}u_\tau(\cdot, t_{k-1}^+)$ is also known (as discussed above), as well as $(\Pi u)(\cdot,t_{k-1}^+)=u(\cdot,t_{k-1})$.
  Thus, we need to solve \eqref{eq:coll_def_w_col} 
  for $t_k^\ell$ where $\ell=1,\ldots, m$, i.e.,
  in contrast with the case $\theta_0>0$, we now have a system of $m$ equations on each interval $(t_{k-1},t_k)$.
  The latter system can be again rescaled to the form
  \eqref{eq:problemRescaled}, only now $U(\cdot, 0)$ is known, so \eqref{eq:problemRescaled}  is restricted to  $\ell\in\{1,\ldots, m\}$
  with unknowns $\{U(\cdot, \theta_{\ell}),\,\ell\ge 1\}$.

  The evaluations in Section~\ref{ssec_matrix} mostly apply to this case,
  only now \new{we have} $v_1(x)=U(x,0)$ in \eqref{matr1},
  the first row in $W$
  in \eqref{matr2} becomes $[1,0,\cdots, 0]$, and the first element in $D_1$ in \eqref{matr3} becomes $\theta_0^\alpha=0$.
  Thus, while \eqref{VV_version} is true, the first element $v_1(x)$ of $\vec V(x)$ is known and should be eliminated, which leads to the following version
  of~\eqref{VV_version}
  for $\hat{V}(x):=\big(v_1(x),\cdots,v_m(x)\big)^{\!\!\top}$, with an appropriate right-hand side $\hat{F}(x)$:
  \begin{equation}\label{VV_version_new}
    \Bigl(\widetilde W +\tau^\alpha \LL \,\hat D_1\widetilde W  \hat D_2\Bigr)\,\hat{V}(x)=\hat{F}(x).
  \end{equation}
  Here the following matrices in $\R^{m,m}$ are used (compare with \eqref{matr2} and \eqref{matr3}):
  \[
    \widetilde W:=
         \pmtrx{
                 \theta_1 & \theta_1^2 &\cdots &\theta_1^m\\
                \vdots\ &\vdots\ && \vdots\ \\
                 \theta_m& \theta_m^2 &\cdots & \theta_m^m}=
                 \hat D_3\pmtrx{
                 1&\theta_1 & \cdots &\theta_1^{m-1}\\
                \vdots\ &\vdots\ && \vdots\ \\
                 1&\theta_m &\cdots & \theta_m^{m-1}}=:\hat D_3\hat  W,
  \]
  \[
      \hat D_1:={\rm diag}(\theta_1^{\alpha},\cdots,\theta_m^{\alpha}), \;\; \hat D_2:={\rm diag}(c_1,\cdots,c_m),\;\;
         \hat D_3:={\rm diag}(\theta_1,\cdots,\theta_m).
  \]
  \new{While the diagonal matrix $\hat D_3$ does not appear in \eqref{VV_version_new}, it will be used to reduce the latter to a system of type \eqref{VV_version}.
  Indeed,} it remains to multiply \eqref{VV_version_new} by $\hat D_3^{-1}$,
  using
  $\hat D_3^{-1}\hat D_1\widetilde W=\hat D_1\hat D_3^{-1}\widetilde W= \hat D_1\hat W$,
  which yields
  \[
    \Bigl(\hat W +\tau^\alpha \LL \,\hat D_1\hat W \hat D_2\Bigr)\,\hat{V}(x)=\hat D_3^{-1}\hat{F}(x).
  \]

  This system is exactly of type \eqref{VV_version}, only it involves $\R^{m,m}$ matrices and
  $0<\theta_1<\theta_m\le 1$ (rather than $\R^{m+1,m+1}$ matrices and $0<\theta_0<\theta_m\le 1$ in \eqref{VV_version}), so it shares
  \textit{all well-posedness properties} established in
  Section~\ref{sec:wellposed}, including those of Section~\ref{ssec_eigen} with $M:=\hat W^{-1}\hat D_1^{-1}\hat W\hat D_2^{-1}$.
  We also note that a version of  Theorem~\ref{the_LaxM} (with $D=\mathrm{diag}(\theta_2,\theta_1)$) applies to the case
  \[
    m=2,\qquad \theta_0=0,\qquad 0<\frac{\theta_1}{\theta_2}\leq \frac{1}{1+\alpha},
  \]
  which is similar to the case \eqref{m1_lax} considered in Section~\ref{ssec_Lax}.

  In Figure~\ref{fig:exist_gausslob}
  \begin{figure}
    \begin{center}
      \includegraphics{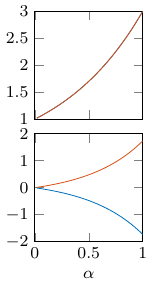}
      \includegraphics{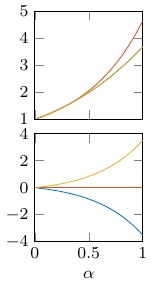}
      \includegraphics{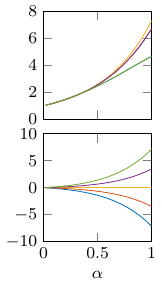}
      \includegraphics{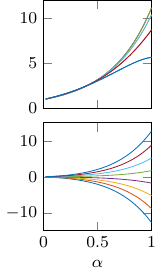}
    \end{center}\vspace{-0.6cm}
    \caption{Eigenvalues for collocation methods using \new{Gauss}-Lobatto points for $m\in\{2,3,5,8\}$ (left to right) and real parts (top), imaginary parts (bottom)\label{fig:exist_gausslob}}
  \end{figure}
  we plot the eigenvalues of $M=\hat W^{-1}\hat D_1^{-1}\hat W\hat D_2^{-1}$ \new{computed} for all $\alpha\in(0,1]$ and $m\in\{2,3,5,8\}$ for the \new{Gauss}-Lobatto collocation points
  (which include $\theta_0=0$ and $\theta_m=1$, i.e. yield continuous collocation schemes).
  We observe that in all cases there are no negative real eigenvalues, and, in fact, the real parts of all eigenvalues are positive.
  These numerical results, clearly, support a version of Theorem~\ref{thm:existCollDisc} for the case $\theta_0=0$.

\section{Residual-type a-posteriori error estimation and adaptive time stepping}\label{sec:res}
  The well-known a-priori error analyses for the popular L1 method show that
  optimal convergence rates (in the context of graded temporal meshes) require strongly-refined temporal meshes at initial time
  \cite{SORG17,Liao_etal_sinum2018,KopMC19}, while even stronger mesh refinements are required for quadratic-interpolation-based methods \cite{Kopteva_Meng,Kopteva2021}.
  Note also that any a-priori error analysis for higher-order
  interpolation-based methods---such as collocation methods considered in this paper---be\-comes very problematic.

  An alternative approach, which we now explore, is to employ adaptive time stepping algorithms based on rigourous a-posteriori error bounds
  \cite{Kopteva22,Kopt_Stynes_apost22,FrK23}.
  \new{This yields temporal grids that are not only} appropriately adapted for initial solution singularities, but also for certain discontinuities in the
  right-hand sides \cite{FrK25}.
  This approach is based on residual-type a-posteriori error estimates, i.e. certain pointwise-in-time errors of the computed solutions, measured in the spatial $L_2(\Omega)$
  or $L_\infty(\Omega)$ norms, are bounded using the \new{standard} residual $\RS := \partial_t^\alpha u_\tau + \LL u_\tau -f$ of the computed solution $u_\tau$
  for any semi-discretisation of \eqref{eq:problem} in time, including the collocation schemes~\eqref{eq:coll_def}.

  Note also recent \new{residual-type} a-posteriori error estimates for subdiffusion equations in \cite{BanjaiMakridakis2022},
 \new{which were obtained in the
  $L_2(\Omega)$ and $L_2((0,t);\,H^1_0(\Omega))$ norms for any $t>0$,
  but tested only on a-priori chosen meshes. In comparison to our approach, the error estimation in the $L_2(\Omega)$ norm in \cite{BanjaiMakridakis2022}
relies on Laplace transforms and resolvent bounds for the spatial operator.}

  We start by summarising a few a-posteriori error bounds from  \cite{Kopteva22,Kopt_Stynes_apost22,FrK23}
  in the form of the following two theorems.
  \new{The notation $(\pt_t^\alpha+\lambda)^{-1} v(t)$, for any appropriate $v$, is used for the solution $w$ of $(\pt_t^\alpha+\lambda)w=v$ $\forall\,t>0$ subject to $w(0)=0$.}
  Importantly, when estimating the errors in the spatial norm $L_\infty(\Omega)$
  we now weaken the regularity assumptions  on $u_\tau$ and $u$ (compared to our earlier results, where they were assumed to be in $C^2(\Omega)$ for each $t>0$).

  \begin{theorem}[a-posteriori error estimate]\label{the1}
    Suppose that $p=2$ and $\LL$ in \eqref{eq:problem} satisfies $\langle \LL v,v\rangle\ge \lambda\|v\|_{L_2(\Omega)}^2$ $\forall\,v\in H_0^1(\Omega)$,
    or $p=\infty$ and $\LL[1]=c(x)\ge \lambda$ for some $\lambda\in\R $.
    Additionally, suppose that a unique solution $u$ of \eqref{eq:problem} and its approximation $u_\tau$ are
    in $C([0,T];\,L_p(\Omega)) \cap W^1_\infty(\epsilon,t;\,L_p(\Omega))$ for any $0<\epsilon<t\le T$,
    and also in $H^1_0(\Omega)$ for any $t>0$, while $u_\tau(\cdot,0)= u_0$.
    Then the error of the latter is bounded in terms of its residual $\RS(\cdot,t) =( \partial_t^\alpha u_\tau + \LL u_\tau -f)(\cdot,t)$
    as follows:
    \begin{align}    \label{L2_error}
      \|(u_\tau-u)(\cdot,t)\|_{L_p(\Omega)}&\le (\pt_t^\alpha+\lambda)^{-1}\|\RS(\cdot, t)\|_{L_p(\Omega)}\qquad\forall\, t>0.
    \end{align}
  \end{theorem}

  \begin{proof}
    For $p=2$ this result is given in \cite[Theorem~2.2]{Kopteva22}; see also \cite[Theorem~2.1]{FrK23}. For $p=\infty$ it is given in \cite[Theorem~3.2]{Kopteva22}, and also in \cite[Theorem~2.2]{FrK23}, however, under stronger regularity assumptions on $u_\tau$ and $u$,
    required for the application of the weak maximum principle for the operator $\pt_t^\alpha+\LL$ to a strong solution $u_\tau-u$.
    Now we weaken these assumptions since
    a similar weak maximum principle is also valid for weak solutions in $H^1_0(\Omega)$ \cite[Theorem 3.1]{Kopteva2022}.
  \end{proof}

  \begin{theorem}[a-priori residual barriers]\label{the_g}
    For some $\lambda\in\R $,
    suppose that \mbox{$p=2$} and $\LL$ in \eqref{eq:problem} satisfies $\langle \LL v,v\rangle\ge \lambda\|v\|^2$ $\forall\,v\in H_0^1(\Omega)$,
    or $p=\infty$ and $\LL g\ge \lambda$ subject to $1\le g\le 1+\omega$ in $\Omega$
    for some $g\in H^1(\Omega)$ and $\omega\in \R$.
    Additionally, let any non-negative barrier function $\EE \in W^1_\infty(\epsilon,t)$ for any \mbox{$0<\epsilon<t\le T$},
    subject to  $0\le \lim_{t\to 0^+}\EE(t)<\infty$  and $\omega\, \pt_t^\alpha\EE(t)\ge 0$  $\forall\,t>0$,
    where  $\omega:=0$ if $p=2$.
    Then, under the assumptions on $u_\tau$ and $u$ made in Theorem~\ref{the1}, for the error $e:=u_\tau-u$, one has
    \begin{equation}\label{R_h_bound_new}
      \|\RS(\cdot,t)\|_{L_p(\Omega)}\leq
        \frac{(\new{\partial_t}^\alpha+\lambda)\EE(t)}{1+\omega}\quad\forall\,t>0
      \quad\Rightarrow\quad
      \|e(\cdot,t)\|_{L_\new{p}}\le \EE(t)\quad\forall\,t>0.
    \end{equation}
    Here $\omega:=0$ and $\displaystyle\lambda:=\min_{\Omega}c(\cdot)$ is one admissible pair for $p=\infty$, while $\omega:=0$ is always used for $p=2$.
    Furthermore,
    \begin{subequations}\label{barrier_bounds_alg_new}
      \begin{align}
        &
        \|\RS(\cdot,t)\|_{L_p(\Omega)}\leq \frac{TOL\cdot\RR_0(t)}{1+\omega}\;\;\;\forall\,t>0
        &\Rightarrow\quad&
        \|e(\cdot,t)\|_{L_p(\Omega)}\leq TOL,
        \label{barrier_bounds_alg_new_a}
        \\[0.3cm]
        &
        \|\RS(\cdot,t)\|_{L_p(\Omega)}\leq \frac{TOL\cdot\RR_1(t)}{1+\omega}\;\;\;\forall\,t>0
        &\Rightarrow\quad&
        \|e(\cdot,t)\|_{L_p(\Omega)}\leq TOL\cdot  t^{\alpha-1},
        \label{barrier_bounds_alg_new_b}
      \end{align}
    \end{subequations}
    where
    \[
      \RR_0(t)=\frac{t^{-\alpha}}{\Gamma(1-\alpha)}+\lambda,
      \,\,\,
      \RR_1(t)=\frac{t^{-1}\bigl[t^{1-\alpha}-((t-\tau)^+)^{1-\alpha}\bigr]}{\Gamma(1-\alpha)\, \tau^{1-\alpha}}+\lambda\,\max\{\tau, t\}^{\alpha-1},
    \]
    with an arbitrary parameter $0<\tau\le t_1$ and the 
    notation \mbox{$(\cdot)^+:=\max\{\cdot,\,0\}$}.
  \end{theorem}

  \begin{proof}
    The bound~\eqref{R_h_bound_new} for $p=2$ is given in \cite[Corollary~2.3]{Kopteva22};
    see also \cite[Corollary~2.4]{FrK23}, as well as for $p=\infty$ and $\omega=1$ (which corresponds to $g:=1$).
    In the more general form, involving $\omega\ge 1$, this bound for $p=\infty$ can be found in \cite[Lemma~2.5]{FrK23}.
    The bounds~\eqref{barrier_bounds_alg_new} follow from \eqref{R_h_bound_new}, as shown in \cite{Kopteva22,FrK23}.
    In comparison to these earlier results,  we now weaken the regularity assumptions on $u_\tau$ and $u$ along the lines of Theorem~\ref{the1}.
  \end{proof}

\subsection{Residuals of the computed solution for collocation schemes \eqref{eq:coll_def}}

  The a-posteriori error bounds \eqref{barrier_bounds_alg_new} very naturally lead to time stepping algorithms with a local time stepping criterion.
  Indeed, once we choose a desired pointwise-in-time error profile---such as $TOL$ in
  \eqref{barrier_bounds_alg_new_a} or $TOL\cdot t^{\alpha-1}$ in
  \eqref{barrier_bounds_alg_new_b}---it suffices to choose each time step $\tau_k$ adaptively such
  that $\|\RS(\cdot,t)\|_{L_p(\Omega)}$ does not exceed the corresponding residual barrier from
  \eqref{barrier_bounds_alg_new} $\forall\,t\in(t_k-1,t_k]$; see \cite[Algorithm 1]{FrK23} for details.

  Thus, an application of such a time stepping algorithm to any particular method hinges on the computation of its residuals.
  For any collocation scheme of type \eqref{eq:coll_def}, whether it is continuous or discontinuous, the residual is given by
  \begin{align}
    \RS &= \partial_t^\alpha u_\tau + \LL u_\tau -f\notag\\
      &= w_\tau +\LL J_t^\alpha w_\tau -f + \LL u_0,\label{eq:residual}
  \end{align}
  where we used $w_\tau$ from \eqref{eq:coll_def_w}. In view of the latter, one immediately concludes that
  \begin{equation}\label{res_zero}
    \RS=0\quad \mbox{in~}\Omega\times\{t_k^\ell\}.
  \end{equation}
  Hence, when applying \eqref{barrier_bounds_alg_new} in adaptive time stepping algorithms,
  it is important to compute $\|\RS(\cdot, t)\|_{L_p(\Omega)}$ at a number of sampling points $t\not\in\{t_k^\ell\}$ on each time interval.
  \new{In all our experiments, 20 sampling points per interval were used for all
  $m\leq 8$
  (while our experiments also show that 10-20 sampling points per time interval suffice for all $m\le 8$).
  More generally, 
  as discussed in \cite[Section~4]{FrK23} (see a separate subsection on {\it Sampling points}),
  the number of sampling points per time interval should be taken larger for higher-order discretizations,
  and it is also helpful to use a graded distribution of these points on each time interval (which reflects a non-symmetric bubble shape of $\|\RS(\cdot, t)\|_{L_p(\Omega)}$).}

  Note also that \eqref{res_zero} implies that the interpolant $\Pi\RS=0$ so \new{we have} $\RS=(I-\Pi)\RS$.
  The latter observation was also valid for another class of collocation schemes considered in
  \cite{FrK23} (in which $u_\tau$ was piecewise polynomial in time, rather than $\pt_t^\alpha u_\tau$)
  and it, in fact, helped to simplify the computation of residuals.
  For the collocation schemes that we consider in this paper, using \eqref{eq:residual}, in which
  $(I-\Pi)w_\tau=0$, one gets
  \begin{align}\notag
    \RS=(I-\Pi)\RS
      &= (I-\Pi)(\LL J_t^\alpha w_\tau -f)\\
      &= (I-\Pi)(\LL u_\tau-f),
      \label{residual_coll}
  \end{align}
  which does not seem easier to compute compared with the original form \eqref{eq:residual}.
  Thus, \new{\eqref{eq:residual}} will be used in all our numerical experiments below.

  \begin{remark}[non-smooth initial condition]\label{rem_u0_nonsmth}
    One apparent difference between formulations \eqref{eq:coll_def1} and \eqref{eq:coll_def_w}
    (or, equivalently,  between \eqref{eq:coll_def1_col} and \eqref{eq:coll_def_w_col}) is in
    that the latter involves $\LL u_0$, and, thus, requires a smoother initial condition $u_0$.
    It should also be noted that while \eqref{eq:coll_def1} yields computed solutions
    $u_\tau(\cdot, t_k^\ell)$ at $\{t_k^\ell\}$, which are solutions of elliptic differential
    equations with the operator $\LL$ (as becomes obvious from rewriting \eqref{eq:coll_def1_col}
    in the form \eqref{VV_version}), the smoothness of $u_\tau(\cdot, t)$ for $t\neq t_k^\ell$ is no better than the smoothness of $u_0$.
    In the simplest case  of $m=0$ and $\theta_0=1$
    this is evident from the representation of
    \[
      u_\tau(\cdot,t)=\sum_{k=0}^M u_\tau(\cdot, t_k)\,\phi_k(t)
    \]
    as a linear combination of the basis functions $\phi_k(t)$, see \cite[Section 3.3 and Fig.\,1]{LiSal23},
    in view of $\phi_0(t)\in C[0,T]$ being strictly positive on each interval $(t_{k-1},t_k)$ $\forall\,k \ge 1$
    (while, as expected, $\phi_0$ vanishes $\forall\,t\in\{t_k\}_{k>0}$).
    Thus any non-smoothness in $u_0$ propagates to the computed solution for
    $t\in(0,T]\backslash\{t_k^\ell\}$---as well as to its residual, which includes $\LL u_\tau$---unless the collocation scheme is modified
    as, for example, described in Remark~\ref{rem_u0_fix}.
  \end{remark}

  \begin{remark}[modified collocation scheme for non-smooth initial condition]\label{rem_u0_fix}
    As explained in Remark~\ref{rem_u0_nonsmth}, any non-smoothness in $u_0$ propagates to the computed solution for $t\in(0,T]\backslash\{t_k^\ell\}$. If undesirable, this
    propagation may be avoided, if
    the considered collocation scheme is modified on the initial time interval $(0,t_1)$  by
    a time stepping scheme based on a piecewise polynomial approximation.
    For example, one can employ the popular L1 scheme or any scheme from \cite{FrK23,FrK24b}, or even the following L0 scheme, which assumes that
    $u_\tau(\cdot, t)=u_\tau(\cdot, t_1)$ for $t\in(0,t_1]$ such that $\pt_t^\alpha u_\tau(\cdot, t)=\Gamma(1-\alpha)^{-1}[u_\tau(\cdot, t_1)-\new{u_0(\cdot)}]t^{-\alpha}$
    for $t\in(0,t_1]$. It follows that
    \[
      \bigl[\Gamma(1-\alpha)^{-1}t_1^{-\alpha}+\LL \bigr] \,u_\tau(\cdot, t_1)=\Gamma(1-\alpha)^{-1} \new{u_0(\cdot)} \,t_1^{-\alpha}+f(\cdot, t_1).
    \]
    Once $u_\tau(\cdot, t)$ is piecewise-polynomial on $(0,t_1]$, this ensures that $u_\tau(\cdot, t_1)$ is a solution of an elliptic equation,
    while, starting from the second time interval, one employs a version of \eqref{eq:coll_def}:
    \[
      \pt_t^{\alpha}u_\tau+ \Pi (\LL u_\tau- f)=0\qquad \mbox{in~}\Omega\times(t_1,T].
    \]
    \new{Note that since $\pt_t^{\alpha}u_\tau$ remains piecewise-polynomial for $t>t_1$, the evaluation of the residual remains unchanged for $t>t_1$,
    i.e. \eqref{residual_coll} remains valid for $t>t_1$. For the first time interval, the evaluation of the residual will depend on the discretization used on this interval.
    For example, if the L1 method is used, for $t\le t_1$ the residual becomes
    \[
      \RS(\cdot,t)=[u_\tau(\cdot,t_1)-u_0(\cdot)]\{\Gamma(2-\alpha)\}^{-1}\tau_1^{-1}t^{1-\alpha}+\LL u_\tau(\cdot,t)-f(\cdot,t)
    \]
    (as elaborated in \cite[Section~3.1]{FrK23}), while $u_\tau(\cdot,t_1)$ is obtained from $\RS(\cdot,t_1)=0$.
    If the L0 scheme is used  (see \cite[Section~4.2]{Kopteva22}),
    then the residual changes to
    \[
      \RS(\cdot,t)=[u_\tau(\cdot,t_1)-u_0(\cdot)]\{\Gamma(1-\alpha)\}^{-1}t^{-\alpha}+\LL u_\tau(\cdot,t)-f(\cdot,t)
    \]
    for $t\le t_1$, while $u_\tau(\cdot,t_1)$ is again obtained from $\RS(\cdot,t_1)=0$.
    }
  \end{remark}

\subsection{Computationally stable implementation}\label{ssec:implementation}

  As was discussed at the beginning of Section~\ref{sec:res}, interpolation-based methods for subdiffusion equations require
  strong local refinements of temporal grids to attain optimal convergence rates.
  By contrast, there is a perception in the community that strongly refined temporal grids, in the context of
  subdiffusion equations, lead to numerically unstable implementations.

  In~\cite[Section~4]{FrK23} we have elaborated on why certain direct implementations may, indeed, lead to numerical instabilities,
  and, furthermore, showed that appropriate treatments of problematic terms  yield
  computationally stable implementations for a number of discretisations
  (an  L1-2 method and a certain class of collocation methods of arbitrary order were addressed, as well as the L1 method).
  Since a stable implementation is an important aspect of any numerical method, we shall now briefly describe 
  how the ideas of \cite{FrK23} apply to the class of collocation schemes considered in this paper.

  Any computationally stable implementation in the context of subdiffusion equations requires appropriate remedies
  for two potential difficulties: i. weakly singular integrals, for which standard quadrature routines become very
  inaccurate; ii. computing differences of nearly equal numbers.
  We shall now discuss where these difficulties occur and how they may be treated.

  Essentially, we need to accurately compute  $J_t^\alpha$ applied to any function from the space $\W_\tau$ of piecewise-polynomial functions.
  For example, to implement \eqref{eq:coll_def_w_col}, one needs to compute $J_t^\alpha w_\tau$ at all collocation points $\{t_k^\ell\}$
  (which is also needed to compute $ u_\tau=u_0+J_t^\alpha w_\tau$ at the collocation points).
  For the adaptive time stepping algorithm, on the other hand, one needs to compute the residual $\RS$ at a number of sample points different from $\{t_k^\ell\}$
  (in view of $\RS(\cdot, t_k^\ell)=0$); hence, one also needs to compute $J_t^\alpha w_\tau$
  (and, thus, by \eqref{eq:residual}, the residual $\RS$)
  at a number of sample points $\not\in\{t_k^\ell\}$ on each $(t_{k-1},t_k)$.

  Consider any polynomial basis, denoted by $\{\hat\phi_\ell(\sigma)\}_{\ell=0}^m$,
  on the reference interval $(0,1)$.
  Now, let the local basis $\{\phi_\ell^k(t)\}_{\ell=0}^m$ with support on $(t_{k-1},t_k)$
  be obtained by the standard linear transformation, i.e. $\phi_\ell^k(t):=\hat\phi_\ell(t_{k-1}+\sigma\tau_k)$ $\forall\,\sigma\in(0,1)$.

  First, we describe the stable computation of $J_t^\alpha\phi_\ell^k$ at
  any point $t=t_{k-1}+\theta\tau_{k}\in(t_{k-1},t_k]$ for any $\theta\in(0,1]$, including $\theta\in\{\theta_j\}$.
  A calculation shows that
  \begin{align*}
    t=t_{k-1}+\theta\tau_{k}
    \quad\Rightarrow\quad
    J_t^\alpha\phi_\ell^k(t)
    &= \frac{1}{\Gamma(\alpha)}\int_{t_{k-1}}^{t}\!(t-s)^{\alpha-1}\,\phi_\ell^k(s)\,\ds\\
    &= \frac{\tau_k^\alpha}{\Gamma(\alpha)}\int_0^{\theta}\!(\theta-\sigma)^{\alpha-1}\,\hat\phi_\ell(\sigma)\,\dsigma.
  \end{align*}
  While the resulting integral is only weakly singular, the reader is cautioned against
  applying standard (typically adaptive) quadrature routines.
  (For example, a simple computational test shows that the Matlab function \texttt{integral} becomes appallingly inaccurate when applied
  to a simple singular integral $\int_0^1 s^{-\alpha}ds$ as $\alpha\rightarrow 1^-$.)
  This difficulty is immediately remedied by applying integration by parts to the above integral, which yields
  \begin{align*}
    J_t^\alpha\phi_\ell^k(t)
      &=\frac{\tau_k^\alpha}{\alpha\,\Gamma(\alpha)}
        \left( \theta^\alpha\,\hat\phi_\ell(0)+\int_0^{\theta}\!(\theta-\sigma)^\alpha\,\hat\phi_\ell'(\sigma)\,\dsigma \right)\\
      &=\tau_k^\alpha\,\frac{\theta^\alpha}{\Gamma(\alpha+1)}
        \left(\hat\phi_\ell(0)+\theta\int_0^1\!(1-\sigma)^\alpha\,\hat\phi_\ell'(\sigma\theta)\,\dsigma \right),
  \end{align*}
  where the remaining integral is no longer singular, so can be computed by an adaptive quadrature rule,
  which is computationally stable and accurate for such integrals.
  Note that further application of integration by parts to evaluate it exactly is not recommended, as this can lead
  to unstable evaluations, as discussed in~\cite{FrK23}.
  When choosing  the basis $\{\hat\phi_{\ell}\}$, one may additionally impose that $\hat\phi_\ell(0)=0$  $\forall\,\ell>0$,
  in which case the term $\hat\phi_\ell(0)$ will vanish $\forall\,\ell>0$.
  Note also, that the dependence on $k$ is only in the factor $\tau_k^\alpha$, so these integrals
  can be assembled in advance for all $\theta\in\{\theta_j\}$, as well as for the predefined set of values of $\theta$ corresponding to the residual sample points on each $(t_{k-1},t_k]$.

  It remains to consider the computation of $J_t^\alpha\phi_\ell^k$ at
  any point $t>t_{k}$, or, equivalently,
  $t=t_{k-1}+\Theta\tau_{k}$ with arbitrarily large $\Theta>1$,
  for which we have
  \begin{align*}
    J_t^\alpha\phi_\ell^k(t)
      &=\frac{1}{\Gamma(\alpha)}\int_{t_{k-1}}^{t_k}\!(t-s)^{\alpha-1}\,\phi_\ell^k(s)\,\ds\\
      &= \frac{\tau_k^\alpha}{\Gamma(\alpha)}\int_0^1\!(\Theta-\sigma)^{\alpha-1}\,\hat\phi_\ell(\sigma)\,\dsigma.
  \end{align*}
  Here note that even if $\hat\phi_\ell(\sigma)=\hat\phi_\ell(1)$, i.e. $\hat\phi_\ell$ is constant,
  the above integral $\int_0^1 \cdots \,\dsigma$ becomes $\new{-}\hat\phi_\ell(1)\, \alpha^{-1}(\Theta-\sigma)^{\alpha}\big|^1_{\sigma=0}$,
  which, for large values of $\Theta$, requires computing the difference of nearly equal numbers.
  To avoid round-off errors, the latter can be implemented using the Matlab commands
  \texttt{expm1} and \texttt{log1p}, respectively corresponding to the functions $\exp(\cdot )-1$
  and $\ln(1+\cdot)$ (or similar commands in  other scientific programming languages)---as elaborated in \cite[see (19)]{FrK23}.
  More generally, a polynomial $\hat\phi_\ell(\sigma)=\hat\phi_\ell(1)-[\hat\phi_\ell(1)-\hat\phi_\ell(\sigma)]$ leads to
  \begin{align*}
    J_t^\alpha\phi_\ell^k(t)
      &= \frac{\tau_k^\alpha}{\Gamma(\alpha)}
          \left(
            \new{-}\hat\phi_\ell(1)\,\frac{(\Theta-\sigma)^{\alpha}}{\alpha}\bigg|^1_{\sigma=0}
            \!\!-\!\int_0^1\!(\Theta-\sigma)^{\alpha-1}\,\bigl[\hat\phi_\ell(1)-\hat\phi_\ell(\sigma)\bigr]\,\dsigma
          \right)\!,
  \end{align*}
  where the integral is non-singular (whether $\Theta$ is close to $1$ or arbitrarily large), so \new{it} can be computed using any quadrature routine, while
  if $\hat\phi_\ell(1)\neq0$,
  the first term requires a stable implementation using \texttt{expm1} and \texttt{log1p}, as we already discussed.

  Thus, the considered collocation schemes can be implemented avoiding any noticeable round-off errors for any $m\ge 0$ and any set of collocation points,
  and, likewise, the computation of residuals remains stable and reliable even for extremely strong local refinements of temporal grids.

\section{Numerical results}\label{sec:numerics}

  As already discussed,
  the a-posteriori error estimates \eqref{barrier_bounds_alg_new}
  provide a rigorous justification for time stepping algorithms with a local time stepping criterion (while the original subdiffusion equations are non-local).
  For example, if we choose a constant value $TOL$ as a desired pointwise-in-time error profile---which corresponds to
  \eqref{barrier_bounds_alg_new_a}---each time step $\tau_k$ is to be chosen adaptively subject to
  $\|\RS(\cdot,t)\|_{L_p(\Omega)}$ not exceeding the corresponding residual barrier from
  \eqref{barrier_bounds_alg_new_a} $\forall\,t\in(\new{t_{k-1}},t_k]$. We refer the reader to \cite[Algorithm 1]{FrK23}
  for details and a discussion of more efficient strategies for choosing the local time steps.
  Note also that the considered collocation schemes, in combination with the time stepping algorithm,
  can be implemented in a computationally stable way---along the lines of Section~\ref{ssec:implementation}---for any fractional order $\alpha$
  in the range between $0.1$ and $0.999$ (at least), and for values of $TOL$
  as small as $10^{-8}$ in Matlab.

  As a test problem, we consider equation \eqref{eq:problem}
  in the domain $\Omega\times[0,T]=(0,1)\times[0,1]$, with $\LL:=-\pt_x^2$, i.e.
  \begin{equation}\label{test}
    \pt^\alpha_t u-\pt_x^2 u=f(x,t)\qquad\mbox{for}\;\;(x,t)\in(0,1)\times(0,1],
  \end{equation}
  subject to the homogeneous boundary conditions in $x$ and a given initial condition.
  Note that one immediately constructs a function $g=1+\frac12\lambda\, x(1- x)$,  associated with this problem,
  that satisfies the conditions of Theorem~\ref{the_g}
  with an arbitrary positive $\lambda:=\LL g=\lambda$ and the corresponding $\omega:=\frac18\lambda$;
  so each such pair $(\lambda,\omega)$ can be used in \eqref{barrier_bounds_alg_new} if the errors are estimated in the $L_\infty(\Omega)$ norm.

  For test problem \eqref{test}, we employ the time stepping algorithm  \cite[Algorithm 1]{FrK23} in all our experiments,
  with the residual profile $TOL\cdot\RR_0(t)/(1+\omega)$, where $\RR_0(t)=t^{-\alpha}/\Gamma(1-\alpha)+\lambda$,
  imposed on the pointwise-in-time residual $\|\RS(\cdot,t)\|_{L_\infty(\Omega)}$.
  In view of \eqref{barrier_bounds_alg_new_a}, this guarantees that for all $t>0$ the error
  $\|e(\cdot,t)\|_{L_\infty(\Omega)}\le TOL$. Thus, the algorithm guarantees that the $L_\infty(0,T;\, L_\infty(\Omega))$ errors will not exceed $TOL$
  in all our computations.

  We consider \new{three} particular cases of \eqref{test}, \new{each} with a known solution.

  \example \label{Ex1}
  Let 
  \[
    u(x,t)=(t^\alpha-t^2+1)\,x\,(1-x)
  \]
  be the exact solution of \eqref{test} with the initial condition $u_0(x)=x\,(1-x)$
  \new{and the right-hand side $f(x,t)=\left(\Gamma(\alpha+1)-\frac{2}{\Gamma(3-\alpha)}t^{2-\alpha}\right)x(1-x) + 2(t^{\alpha}-t^2+1)$.
  }%
  The exact solution exhibits a typical weak singularity of type $t^\alpha$ near $t=0$.
  Note that $u$ is a quadratic polynomial as a function of $x$ for each $t\ge 0$. Therefore, using
  piecewise quadratic elements in space, on a coarse spatial grid of just ten cells
  resolves it exactly, and any error obtained is purely due to the time
  discretisation.

  We begin our numerical investigation of collocation schemes~\eqref{eq:coll_def}, both discontinuous and continuous,
  by comparing various sets of collocation points for a
  fixed polynomial degree of $m=4$.
  Figure~\ref{fig:R0LinfLinf_04_cmp}
  \begin{figure}[htb]
    \begin{center}
      \includegraphics{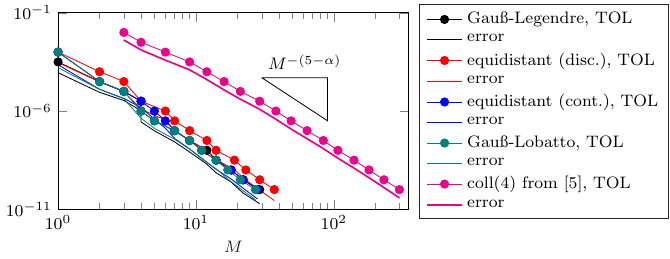}
    \end{center}\vspace{-0.6cm}
    \caption{Adaptive time stepping algorithm for Test example~\ref{Ex1}:
             $TOL$ and the corresponding $L_\infty(0,T;\, L_\infty(\Omega))$ errors
             vs. number of time intervals $M$ for $m=4$ and various choices of collocation points, $\alpha=0.4$,
             residual barrier $\RR_0$ with $\lambda=\pi^2$ and $\omega=\lambda/8$.\label{fig:R0LinfLinf_04_cmp}}
  \end{figure}
  shows the results of the time stepping adaptation. 
  While the number of time intervals $M$ is chosen adaptively for each prescribed $TOL$, we plot both
  prescribed tolerances $TOL$ and the corresponding $L_\infty(0,T;L_\infty(\Omega))$ errors as functions of $M$.
  We ran experiments with  \eqref{eq:coll_def} for four sets of collocation points.
  Two discontinuous schemes were considered, with
  i. \new{Gauss}-Legendre points, ii. equidistant points ($\theta_j=\frac{j+1}{m+2}$), as well as
  two continuous schemes, with  iii. equidistant points ($\theta_j=\frac{j}{m+1}$), iv. \new{Gauss}-Lobatto points.
  Additionally, in the same Figure~\ref{fig:R0LinfLinf_04_cmp} we present analogous results for the collocation method of order $m=4$ considered in \cite{FrK23}
  (in which the computed solution is a continuous piecewise polynomial of order $m$) with equidistant collocation points ($\theta_j=\frac{j}{m+1}$).

  In all these experiments we observe that the errors are, as expected, below the corresponding $TOL$.
  Furthermore, the errors remain quite close to $TOL$ in all our experiments, which indicates good effectivity
  indices of the considered a-posteriori error estimates, and also illustrates the efficiency of the time stepping algorithm.

  Within the same class of collocation schemes  \eqref{eq:coll_def}, the considered sets collocation points
  showed very slight error variations in Figure~\ref{fig:R0LinfLinf_04_cmp},
  with \new{Gauss}-Legendre points performing slightly better.
  By contrast, it is clear from this figure that for Test example~\ref{Ex1}, collocation methods of type
  \eqref{eq:coll_def}, with piecewise polynomial fractional derivatives $\pt_t^\alpha u_\tau$ of the computed
  solution $u_\tau$, noticeably outperform those from \cite{FrK23}, with piecewise polynomial $u_\tau$.
  While all considered five methods exhibit close to optimal convergence rate $\ord{M^{-(5-\alpha)}}$,
  we observe that the collocation methods \eqref{eq:coll_def} require substantially smaller values of $M$ to achieve
  the same level of accuracy.
  This phenomenon is easily explained by that the error in \eqref{eq:coll_def} is induced by  a polynomial approximation of $\pt_t^\alpha u$;
  hence, the accuracy depends on the smoothness of $\pt_t^\alpha u$ (rather than on the smoothness of $u$, which is the case for the collocation methods in \cite{FrK23}).
  Although the solution $u$ of Test example~\ref{Ex1} has
  a weak initial singularity of type $t^\alpha$, one immediately notes a substantially weaker singularity  in $\partial_t^\alpha u$ of type $t^{2-\alpha}$.
  This is illustrated in Figure~\ref{fig:Ex1_method_cmp},
  \begin{figure}[htb]
    \begin{center}
      \includegraphics{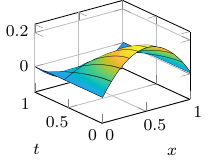}
      \includegraphics{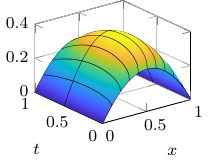}
      \includegraphics{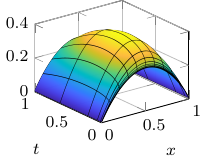}
    \end{center}\vspace{-0.6cm}
    \caption{Test example~\ref{Ex1}, $\alpha=0.4$, adaptive computed solutions, $TOL=10^{-4}$,
             residual barrier $\RR_0$ with $\lambda=\pi^2$ and $\omega=\lambda/8$:
             $\pt_t^\alpha u_\tau$ (left) and $u_\tau$ (center),
             as well as the collocation solution from \cite{FrK23} (right).\label{fig:Ex1_method_cmp}}
  \end{figure}
  which clearly shows  that  $\partial_t^\alpha u$ in Test example~\ref{Ex1} has no visible initial singularity.
  Consequently, the resulting temporal meshes  can be much coarser near $t=0$ compared to the
  collocation scheme from \cite{FrK23}.

  Importantly, the adaptive algorithm automatically identifies that a relatively \new{coarse} temporal grid suffices
  if \eqref{eq:coll_def} is used for Test example~\ref{Ex1}, i.e. the local time step criterion automatically
  adjusts to the error of a given numerical method applied to a given problem.

  In the next Figure~\ref{fig:R0LinfLinf_04_GLeg}, for the same Test example~\ref{Ex1},
  we vary the polynomial degree $m$ between $0$ and $8$, while using \new{Gauss}-Legendre collocation points.
  \begin{figure}[htb]
    \begin{center}
      \includegraphics{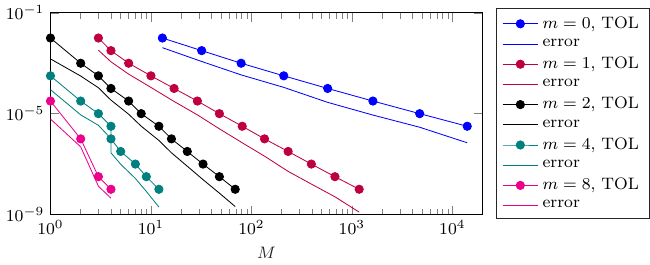}
    \end{center}\vspace{-0.6cm}
    \caption{Adaptive time stepping algorithm for Test example~\ref{Ex1}: $TOL$ and the corresponding
             $L_\infty(0,T;\, L_\infty(\Omega))$ errors vs. number of time steps $M$
             for various polynomial degrees $m$ and  \new{Gauss}-Legendre collocation points, $\alpha=0.4$,
             residual barrier $\RR_0$ with $\lambda=\pi^2$ and $\omega=\lambda/8$.\label{fig:R0LinfLinf_04_GLeg}}
  \end{figure}
  As expected, as $m$ increases, the method yields smaller errors and higher convergence rates.
  Here we even achieve an error below the $TOL=10^{-8}$ bound for $m=8$ with $M=4$, while with $M=10^4$, the piecewise-constant
  collocation scheme, with $m=0$, yields
  errors just below the $TOL=10^{-5}$ bound. We also observe improvements in computation times for
  the entire adaptive process with higher-order collocation schemes \eqref{eq:coll_def}.
  \new{
  This can be seen clearly in Table~\ref{tab:speed},
  \begin{table}[htb]
  \new{
    \caption{Number of cells and adaptation runtime for Test example~\ref{Ex1} for two
             different given $TOL$ and various polynomial degrees $m$, using Gauss-Legendre points, $\alpha=0.4$,
             residual barrier $\RR_0$ with $\lambda=\pi^2$ and $\omega=\lambda/8$.\label{tab:speed}}
    \begin{center}
      \begin{tabular}{crr|rr}
        &\multicolumn{2}{c}{$TOL=10^{-5}$}
        &\multicolumn{2}{c}{$TOL=10^{-8}$}\\
        $m$ & \#cells & time [sec] & \#cells & time[sec]\\
        \hline
        0 & 4729 & 504.5 & \\
        1 &   48 &   2.0 & 1182 & 82.6\\
        2 &    8 &   0.7 &   69 &  3.1\\
        4 &    3 &   0.5 &   12 &  1.7\\
        8 &    2 &   0.4 &    4 &  1.2
      \end{tabular}
    \end{center}
    }
  \end{table}
  where for various polynomial degrees and two values of $TOL$ the number of cells and computational runtime
  for the adaptive method are given.
  For the other examples we observe a similar behaviour.
  }

  \example \label{Ex2}
  Let 
  \[
    u(x,t)=(t^\alpha-t^{2\alpha}+1)\,x\,(1-x)
  \]
  be the exact solution of \eqref{test}
  \new{with the right-hand side
  \[
    f(x,t)=\left(\Gamma(\alpha+1)-\frac{\Gamma(2\alpha+1)}{\Gamma(\alpha+1)}t^{\alpha}\right)x(1-x) + 2(t^{\alpha}-t^{2\alpha}+1).
  \]
  }%
  In contrast to Test example~\ref{Ex1}, both
  $u$ and $\pt_t^\alpha u$ exhibit a weak initial singularity of the same type $t^\alpha$,
  so we expect stronger-refined adaptive temporal grids, and also more similarity in the errors
  generated by \eqref{eq:coll_def} as opposed to collocation schemes from \cite{FrK23}.
  Indeed, Figure~\ref{fig:Ex2_method_cmp}
  \begin{figure}[htb]
    \begin{center}
      \includegraphics{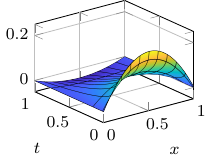}
      \includegraphics{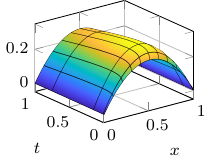}
      \includegraphics{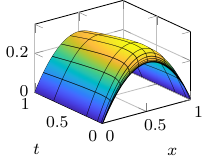}
    \end{center}\vspace{-0.6cm}
    \caption{Test example~\ref{Ex2}, $\alpha=0.4$, adaptive computed solutions, $TOL=10^{-4}$, residual barrier
             $\RR_0$ with $\lambda=\pi^2$ and $\omega=\lambda/8$:
             $\pt_t^\alpha u_\tau$ (left) and $u_\tau$ (center),
             as well as the collocation solution from \cite{FrK23} (right).\label{fig:Ex2_method_cmp}}
  \end{figure}
  shows a weak initial singularity in both $\pt_t^\alpha u_\tau$ and $u_\tau$ and a stronger initial
  refinement of the adaptive temporal grid. At the same time,
  the grid is not as refined as the corresponding one for the collocation scheme of type \cite{FrK23} generated with the same algorithm.

  Next, in Figure~\ref{fig:ex2_R0LinfLinf_04_GLeg}, we again vary the polynomial degree $m$ between $0$ and $8$,
  \begin{figure}[htb]
    \begin{center}
      \includegraphics{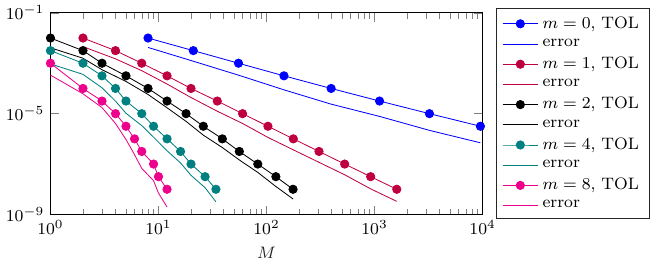}
    \end{center}\vspace{-0.6cm}
    \caption{Adaptive time stepping algorithm for Test example~\ref{Ex2}: $TOL$ and the corresponding
             $L_\infty(0,T;\, L_\infty(\Omega))$ errors vs. number of time steps $M$
             for various polynomial degrees $m$ and  \new{Gauss}-Legendre collocation points, $\alpha=0.4$,
             residual barrier $\RR_0$ with $\lambda=\pi^2$ and $\omega=\lambda/8$.\label{fig:ex2_R0LinfLinf_04_GLeg}}
  \end{figure}
  while using \new{Gauss}-Legendre collocation points.
  As expected, we again observe higher convergence rates and smaller errors with an increasing polynomial degree $m$.

  Finally, in Figure~\ref{fig:ex2_R0LinfLinf_04_cmp}
  \begin{figure}[htb]
    \begin{center}
      \includegraphics{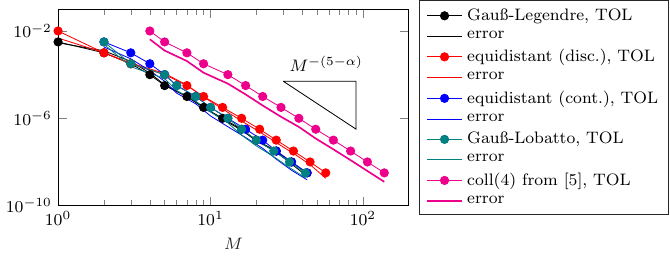}
    \end{center}\vspace{-0.6cm}
    \caption{Adaptive time stepping algorithm for Test example~\new{\ref{Ex2}}: $TOL$ and the corresponding
             $L_\infty(0,T;\, L_\infty(\Omega))$ errors vs. number of time steps $M$
             for $m=4$, $\alpha=0.4$,
             residual barrier $\RR_0$ with $\lambda=\pi^2$ and $\omega=\lambda/8$,
             \eqref{eq:coll_def} with \new{various choices of collocation} points and the one from \cite{FrK23}.\label{fig:ex2_R0LinfLinf_04_cmp}}
  \end{figure}
  we compare the performance of collocation schemes in \eqref{eq:coll_def} and \cite{FrK23}, both combined with  the same adaptive algorithm,
  and again observe that the class of collocation methods \eqref{eq:coll_def}, considered in this paper, still outperforms the one in \cite{FrK23}
  (albeit not as significantly as in the context of Test example~\ref{Ex1}).

  \new{
  \example \label{Ex3}
  Let
  \[
    u(x,t)=(t^{2\alpha}-t^2+1)\,x\,(1-x)
  \]
  be the exact solution of \eqref{test}
  with the right-hand side
  \[
    f(x,t)=\left(\frac{\Gamma(2\alpha+1)}{\Gamma(\alpha+1)}t^{\alpha}-\frac{2}{\Gamma(3-\alpha)}t^{2-\alpha}\right)x(1-x) + 2(t^{\alpha}-t^{2\alpha}+1).
  \]
  In contrast to the previous examples, the solution $u$ exhibits a weaker singularity than $\partial_t^\alpha u$.
  Again we want to compare the new method with the one in \cite{FrK23}.
  Figure~\ref{fig:Ex3_method_cmp}
  \begin{figure}[htb]
    \begin{center}
      \includegraphics{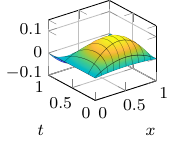}
      \includegraphics{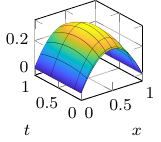}
      \includegraphics{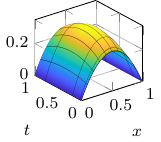}
    \end{center}\vspace{-0.6cm}
    \caption{Test example~\ref{Ex3}, $\alpha=0.4$, adaptive computed solutions, $TOL=10^{-4}$, residual barrier
             $\RR_0$ with $\lambda=\pi^2$ and $\omega=\lambda/8$:
             $\pt_t^\alpha u_\tau$ (left) and $u_\tau$ (center),
             as well as the collocation solution from \cite{FrK23} (right).\label{fig:Ex3_method_cmp}}
  \end{figure}
  shows weak initial singularities in both $\pt_t^\alpha u_\tau$ and $u_\tau$ and also, that the mesh
  adaptation algorithm gives very similar meshes for the collocation method in this paper and the one from
  \cite{FrK23}. Indeed, the direct comparison in Figure~\ref{fig:ex3_R0LinfLinf_04_cmp}
  \begin{figure}[htb]
    \begin{center}
      \includegraphics{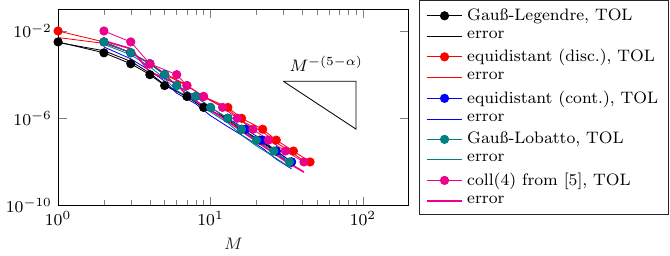}
    \end{center}\vspace{-0.6cm}
    \caption{Adaptive time stepping algorithm for Test example~\ref{Ex3}: $TOL$ and the corresponding
             $L_\infty(0,T;\, L_\infty(\Omega))$ errors vs. number of time steps $M$
             for $m=4$, $\alpha=0.4$,
             residual barrier $\RR_0$ with $\lambda=\pi^2$ and $\omega=\lambda/8$,
             \eqref{eq:coll_def} with {various choices of collocation} points and the one from \cite{FrK23}.\label{fig:ex3_R0LinfLinf_04_cmp}}
  \end{figure}
  for $m=4$ shows a very similar performance of both methods, although again the new method is slighly better.
  This example further strengthens our interpretation, that a weaker singularity leads to a better performance
  of the method.

  Finally, Figure~\ref{fig:ex3_R0LinfLinf_04_GLeg}
  \begin{figure}[htb]
    \begin{center}
      \includegraphics{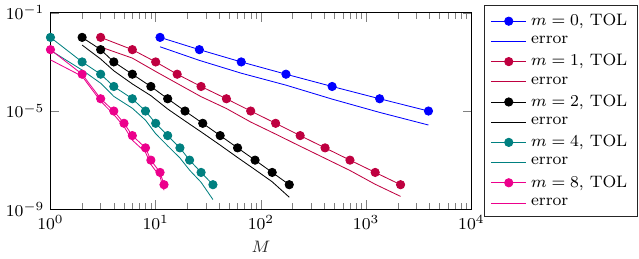}
    \end{center}\vspace{-0.6cm}
    \caption{Adaptive time stepping algorithm for Test example~\ref{Ex3}: $TOL$ and the corresponding
             $L_\infty(0,T;\, L_\infty(\Omega))$ errors vs. number of time steps $M$
             for various polynomial degrees $m$ and {Gauss}-Legendre collocation points, $\alpha=0.4$,
             residual barrier $\RR_0$ with $\lambda=\pi^2$ and $\omega=\lambda/8$.\label{fig:ex3_R0LinfLinf_04_GLeg}}
  \end{figure}
  shows again improved convergence for higher polynomial degrees.
  }
\section*{Conclusions}
  The paper addresses the numerical solution of time-fractional parabolic equations with a Caputo time derivative of order $\alpha\in(0,1)$.
  In the context of such equations, we have considered a class of collocation schemes, which assume that
  the Caputo derivative  of the computed solution (rather than the computed solution itself)
  is a polynomial of degree $m\ge 0$ in time on each time interval.
  For such discretisations of any order $m\ge 0$, with any choice of collocation points,
  we have given sufficient conditions for existence and uniqueness of collocation solutions,
  with both continuous and discontinuous versions having been addressed.
  While any a-priori error analysis for the considered methods is very problematic, we have presented
  residual-type a-posteriori estimates for the errors induced by discretisation in time,
  which naturally lead to adaptive time stepping algorithms, such as developed in \cite{FrK23},
  with local time step criteria (while the original subdiffusion equation is non-local).
  In the context of such adaptive algorithms, we have described computationally stable implementations of the considered collocation schemes and
  performed extensive numerical experiments, which support the applicability and reliability of the constructed time stepping
  algorithm. Furthermore, for both test examples, our numerical results suggest that the considered class of collocation schemes outperforms the one from~\cite{FrK23}.
\new{Finally, the a-posteriori error estimation and
  adaptive time stepping algorithms, such as developed here and in  \cite{Kopteva22,Kopt_Stynes_apost22,FrK23},
 are particularly important for nonlinear time-fractional parabolic equations (which frequently require long-time computations).
While preliminary estimators for this case have been presented in our recent paper \cite[Section 2.2]{FrK23},
 the development of
reliable and efficient time stepping algorithms for semilinear subdiffusion equations is
 part of our ongoing research and will be addressed elsewhere.}

  \bibliographystyle{plainurl}
  \bibliography{lit}
\end{document}